\documentclass[12pts]{article}
\usepackage{amsmath,amsfonts,amssymb,amsthm,mathtools,mathdots}
\usepackage{microtype}
\usepackage{lipsum}
\usepackage{hyperref}
\usepackage{indentfirst}
\hypersetup{colorlinks, linkcolor={red}, citecolor={blue}, urlcolor={black}}
\usepackage{comment}
\usepackage[margin=1.25in]{geometry}
\usepackage{enumerate}
\usepackage{subcaption}
\usepackage{tikz}
\usetikzlibrary{calc,matrix,positioning,angles,quotes}
\usetikzlibrary{shapes.geometric}
\usetikzlibrary{patterns,decorations.pathreplacing}

\allowdisplaybreaks

\newtheorem{theorem}{Theorem}[section]
\newtheorem{lemma}[theorem]{Lemma}
\newtheorem{proposition}[theorem]{Proposition}

\newtheorem{conjecture}[theorem]{Conjecture}

\title{Degree sequences of triangular multigraphs}
\author{John Talbot\thanks{Department of Mathematics, UCL, UK. Email: {\tt j.talbot@ucl.ac.uk}} \and Jun Yan\thanks{Mathematics Institute, University of Warwick, UK. Email: {\tt jun.yan@warwick.ac.uk}. Supported by the Warwick Mathematics Institute Centre for Doctoral Training and funding from the UK EPSRC (Grant number: EP/W523793/1).}}
\date{}

\begin{document}

\maketitle

\abstract{A simple graph is \emph{triangular} if every edge is contained in a triangle. A sequence of integers is \emph{graphical} if it is the degree sequence of a simple graph. Egan and Nikolayevsky \cite{EN} recently conjectured that every graphical sequence whose terms are all at least 4 is the degree sequence of a triangular simple graph, and proved this in some special cases. In this paper we state and prove the analogous version of this conjecture for multigraphs.}

\section{Introduction}
A graph is \emph{simple} if it does not contain any loops or multiple edges. A sequence of integers $(d_1,\cdots,d_n)$ is \emph{graphical} if there exists a simple graph $G$ on vertices $v_1,\cdots,v_n$ such that $\deg(v_i)=d_i$ for all $i\in[n]$. The well-known Erd\H{o}s-Gallai Theorem provides a complete characterisation of graphical sequences. 
\begin{theorem}[Erd\H{o}s-Gallai Theorem \cite{EG}]\label{erdosgallai}
A sequence of positive integers $d_1\geq\cdots\geq d_n$ is graphical if and only if 
\begin{itemize}
\item $d_1+\cdots+d_n$ is even, 
\item $\sum_{i=1}^kd_i\leq k(k-1)+\sum_{i=k+1}^n\min\{d_i,k\}$ for every $k\in[n]$. 
\end{itemize}
\end{theorem}

A \emph{triangle} in a simple graph consists of three distinct vertices that are pairwise adjacent. A simple graph is \emph{triangular} if every edge is contained in a triangle. Recently, Egan and Nikolayevsky \cite{EN} conjectured that any positive integer sequence whose terms are all at least 4, and satisfies the obvious necessary condition of being graphical, is the degree sequence of a triangular simple graph. By the Erd\H{o}s-Gallai Theorem, this is equivalent to the following. 
\begin{conjecture}[Egan and Nikolayevsky \cite{EN}]\label{conj}
If $n\geq 3$ and $(d_1,\cdots,d_n)$ is a sequence of integers satisfying
\begin{itemize}
    \item $d_1\geq\cdots\geq d_n\geq4$,
    \item $d_1+\cdots+d_n$ is even,
    \item $\sum_{i=1}^kd_i\leq k(k-1)+\sum_{i=k+1}^n\min\{d_i,k\}$ for every $k\in[n]$,
\end{itemize} 
then it is the degree sequence of a triangular simple graph. 
\end{conjecture}

Egan and Nikolayevsky \cite{EN} proved this conjecture in the case when the degree sequence contains at most two distinct terms. 
\begin{theorem}[Egan and Nikolayevsky \cite{EN}]
Any graphical sequence of the form $(a^p,b^q)$, where $a>b\geq4$ and $p\geq0,q>0$ is the degree sequence of a triangular simple graph. 
\end{theorem}

In this paper, we state and prove the analogous version of Conjecture \ref{conj} for multigraphs. A \emph{triangle} in a multigraph consists of three distinct vertices which are pairwise adjacent. A multigraph is \emph{triangular} if every edge is contained in a triangle. The following lemma provides two necessary conditions for the degree sequences of triangular multigraphs.  
\begin{lemma}\label{triangularnecessary}
If $n\geq 3$ and $d_1\geq\cdots\geq d_n>0$ is the degree sequence of a triangular multigraph on $n$ vertices, then 
\begin{itemize}
    \item $\sum_{i=1}^nd_i$ is even,
    \item $d_1\leq\sum_{i=2}^n(d_i-1)$. 
\end{itemize}
\end{lemma}
\begin{proof}
$\sum_{i=1}^nd_i$ is even from the well-known handshake lemma. Now suppose for a contradiction that $d_1>\sum_{i=2}^n(d_i-1)$ and $G$ is a triangular multigraph on vertices $v_1,\cdots,v_n$ satisfying $\deg(v_i)=d_i$ for all $i\in[n]$. Note that any triangular multigraph is necessarily loopless. If for every $2\leq i\leq n$, $v_i$ is adjacent to a vertex that is not $v_1$, then since $G$ is loopless, $\deg(v_1)\leq\sum_{i=2}^n(d_i-1)<d_1$, contradiction. Hence, there must exist some $2\leq i\leq n$ such that $v_i$ is only adjacent to the vertex $v_1$. Since $d_i>0$, the edge $v_1v_i$ has positive multiplicity, but cannot be in a triangle, contradicting $G$ is triangular.
\end{proof}


Our  main result is that the analogue of Conjecture \ref{conj} holds for multigraphs. Any sequence of $n \ge 3$ integers, each at least 4, and satisfying the obvious necessary conditions in Lemma \ref{triangularnecessary} is the degree sequence of a triangular multigraph. 
\begin{theorem}\label{main}
If $n\geq3$ and $(d_1,\cdots,d_n)$ is a sequence of integers satisfying
\begin{enumerate}[(i)]
\item $d_1\geq\cdots\geq d_n\geq4$,\label{decreasing}
\item $\sum_{i=1}^nd_i$ is even,\label{even}
\item $d_1\leq\sum_{i=2}^n(d_i-1)$,\label{d1atmost}
\end{enumerate} 
then it is the degree sequence of a triangular multigraph. 
\end{theorem}

As evidenced by the following proposition, we cannot replace the number 4 in condition (\ref{decreasing}) by a smaller integer.

\begin{proposition}
The degree sequence given by $d_i=3$ for all $i\in[n]$ is the degree sequence of a triangular multigraph if and only if $n$ is divisible by 4.
\end{proposition}
\begin{proof}
Let $G$ be a triangular multigraph on $n$ vertices, all of which have degree 3. It suffices to show that every connected component of $G$ is isomorphic to $K_4$. 

Fix a connected component of $G$. Suppose there exists a vertex $v_1$ adjacent to three different vertices $v_2,v_3,v_4$. As edges $v_1v_2,v_1v_3,v_1v_4$ all need to be in triangles, we may, without loss of generality, assume edges $v_2v_3, v_2v_4$ are also in $G$. If edge $v_3v_4$ is also in $G$, then all of $v_1,v_2,v_3,v_4$ have degree $3$, so the connected component containing them is isomorphic to $K_4$. Otherwise, vertex $v_3$ is adjacent to a new vertex $v_5$. But since $v_1v_5,v_2v_5$ are not in $G$, the edge $v_3v_5$ is not in a triangle, contradiction. 

Suppose now there is no vertex in this connected component that is adjacent to three different vertices. Let $v_1$ be a vertex in this component. Either there is an edge $v_1v_2$ of multiplicity 3, which cannot be in a triangle, or we have an edge $v_1v_2$ with multiplicity 2 and an edge $v_1v_3$ with multiplicity 1. For edges $v_1v_2,v_1v_3$ to be in triangles, we must have edge $v_2v_3$ as well. One of edges $v_1v_3,v_2v_3$ must have multiplicity at least 2, as $v_3$ cannot be adjacent to three different vertices. But then one of $v_1,v_2$ will have degree at least 4, contradiction.
\end{proof}

\section{Proof of Theorem \ref{main}}
Suppose $n\ge 3$ and $(d_1,\cdots,d_n)$ is a sequence of integers satisfying (\ref{decreasing})-(\ref{d1atmost}). The goal of Theorem \ref{main} is to construct a triangular multigraph $G$ on vertices $v_1,\cdots,v_n$, such that $\deg(v_i)=d_i$ for all $i\in[n]$. It turns out that $D=\sum_{i=1}^n(-1)^{i-1}d_i$ is a critical quantity that will guide our constructions. Note that $D$ is non-negative by (\ref{decreasing}) and is even by (\ref{even}).

If $D\geq n-2$, we show in Lemma \ref{bigD} that a fan-shaped construction (see Figure \ref{fig:bigD}) works, with $v_1$ being the central vertex. If $D\leq 4$, we show in Lemma \ref{smallD} that a construction based on modifying the square of the length $n$ cycle (see Figure \ref{fig:smallD}) works. Finally, we complete the proof of Theorem \ref{main} by showing that in the intermediate case, $6\leq D\leq n-3$, a combination of the above two constructions, with $v_1$ being the unique common vertex, works. 

In order to combine these two constructions in the proof of Theorem \ref{main}, we will need to state and prove Lemma \ref{bigD} and Lemma \ref{smallD} in the slightly more general setting where we do not assume $d_1$ is the largest term of the sequence. Throughout the constructions in this section, the multiplicity of an edge $v_iv_j$ in a multigraph $G$ will be denoted by $m(v_i,v_j)$.

\begin{lemma}\label{bigD}
Let $n\geq 3$ and let $(d_1,\cdots,d_n)$ be a sequence of non-negative integers satisfying
\begin{itemize}
\item $d_2\geq\cdots\geq d_n\geq4$,
\item $d_1+\cdots+d_n$ is even,
\item $d_1\leq\sum_{i=2}^n(d_i-1)$,
\item $D=\sum_{i=1}^n(-1)^{i-1}d_i\geq n-2$,
\end{itemize}
then there exists a triangular multigraph $G$ with degree sequence $(d_1,\cdots,d_n)$.  
\end{lemma}

\begin{figure}
\centering
\begin{subfigure}{\textwidth}
\centering
\begin{tikzpicture}[scale=
2.5]
\draw ($({-cos(66)},{sin(66)})$) node(v2)[inner sep=0.3ex,circle,fill=black]{};
\node [above] at ($({-cos(66)},{sin(66)})$) {$v_2$};
\draw ($({cos(66)},{sin(66)})$) node(vn)[inner sep=0.3ex,circle,fill=black]{};
\node [above] at ($({cos(66)},{sin(66)})$) {$v_n$};
\draw ($({cos(42)},{sin(42)})$) node(vn-1)[inner sep=0.3ex,circle,fill=black]{};
\node [right] at ($({cos(42)},{sin(42)})$) {$v_{n-1}$};
\draw ($({-cos(42)},{sin(42)})$) node(v3)[inner sep=0.3ex,circle,fill=black]{};
\node [left] at ($({-cos(42)},{sin(42)})$) {$v_3$};
\draw ($({-cos(30)},{-sin(30)})$) node(vn-2k-1)[inner sep=0.3ex,circle,fill=black]{};
\node [left] at ($({-cos(30)},{-sin(30)})$) {$v_{n-2k-1}$};
\draw ($({-cos(54)},{-sin(54)})$) node(vn-2k)[inner sep=0.3ex,circle,fill=black]{};
\node [left] at ($({-cos(54)},{-sin(54)})$) {$v_{n-2k}$};
\draw ($({cos(30)},{-sin(30)})$) node(vn-2k+4)[inner sep=0.3ex,circle,fill=black]{};
\node [right] at ($({cos(30)},{-sin(30)})$) {$v_{n-2k+4}$};
\draw ($({cos(54)},{-sin(54)})$) node(vn-2k+3)[inner sep=0.3ex,circle,fill=black]{};
\node [right] at ($({cos(54)},{-sin(54)})$) {$v_{n-2k+3}$};
\draw ($({cos(78)},{-sin(78)})$) node(vn-2k+2)[inner sep=0.3ex,circle,fill=black]{};
\node [right] at ($({cos(78)},{-sin(78)})$) {$v_{n-2k+2}$};
\draw ($({-cos(78)},{-sin(78)})$) node(vn-2k+1)[inner sep=0.3ex,circle,fill=black]{};
\node [left] at ($({-cos(78)},{-sin(78)})$) {$v_{n-2k+1}$};
\draw (0,0) node(v1)[inner sep=0.3ex,circle,fill=black]{};
\node [right] at (0.1,0) {$v_1$};
\node [left] at (0.75,0.1) {$\vdots$};
\node [left] at (-0.75,0.1) {$\vdots$};

\draw (v2) edge node[above,sloped,text=red,font=\tiny]{} (v1);
\draw (v1) edge node[below,sloped,text=red,font=\tiny]{} (v3);
\draw (v1) edge node[above,sloped,text=red,font=\tiny]{} (vn-2k-1);
\draw (v1) edge node[below,sloped,text=red,font=\tiny]{} (vn-2k);
\draw (vn-2k+1) edge node[above,sloped,text=red,font=\tiny]{} (v1);
\draw (v1) edge node[above,sloped,text=red,font=\tiny]{} (vn-2k+2);
\draw (v1) edge node[below,sloped,text=red,font=\tiny]{} (vn-2k+3);
\draw (v1) edge node[above,sloped,text=red,font=\tiny]{} (vn-2k+4);
\draw (v1) edge node[below,sloped,text=red,font=\tiny]{} (vn-1);
\draw (v1) edge node[above,sloped,text=red,font=\tiny]{} (vn);
\draw (v2) edge node[left=-0.1,anchor=south east,text=red,font=\tiny]{$d_3-1$} (v3);
\draw (vn-1) edge node[right=-0.1,anchor=south west,text=red,font=\tiny]{$1$} (vn);
\draw (vn-2k+3) edge node[right=0.1,text=red,font=\tiny]{$1$} (vn-2k+4);
\draw (vn-2k+1) edge node[below=0.1,text=red,font=\tiny]{$d_{n-2k+2}-1-\delta$} (vn-2k+2);
\draw (vn-2k-1) edge node[left=0.1, text=red,font=\tiny]{$d_{n-2k}-1$} (vn-2k);
\end{tikzpicture}
\caption{$n$ odd}
\label{fig:bigD1}
\end{subfigure}
\begin{subfigure}{\textwidth}
\centering
\begin{tikzpicture}[scale=2.5]

\draw ($({-cos(66)},{sin(66)})$) node(vn)[inner sep=0.3ex,circle,fill=black]{};
\node [left] at ($({-cos(66)},{sin(66)})$) {$v_n$};
\draw ($(0,1)$) node(vn-1)[inner sep=0.3ex,circle,fill=black]{};
\node [above] at ($(0,1)$) {$v_{n-1}$};
\draw ($({cos(66)},{sin(66)})$) node(vn-2)[inner sep=0.3ex,circle,fill=black]{};
\node [right] at ($({cos(66)},{sin(66)})$) {$v_{n-2}$};
\draw ($({cos(18)},{sin(18)})$) node(vn-3)[inner sep=0.3ex,circle,fill=black]{};
\node [right] at ($({cos(18)},{sin(18)})$) {$v_{n-3}$};
\draw ($({cos(6)},{-sin(6)})$) node(vn-4)[inner sep=0.3ex,circle,fill=black]{};
\node [right] at ($({cos(6)},{-sin(6)})$) {$v_{n-4}$};
\draw ($({-cos(18)},{sin(18)})$) node(v2)[inner sep=0.3ex,circle,fill=black]{};
\node [left] at ($({-cos(18)},{sin(18)})$) {$v_2$};
\draw ($({-cos(6)},{-sin(6)})$) node(v3)[inner sep=0.3ex,circle,fill=black]{};
\node [left] at ($({-cos(6)},{-sin(6)})$) {$v_3$};

\node at (0,-0.4) {$\cdots$};
\node at ($({0.4*-cos(48)},{-0.4*sin(48)})$) {$\ddots$};
\node at ($({0.4*cos(48)},{-0.4*sin(48)})$) {$\iddots$};
\draw (0,0) node(v1)[inner sep=0.3ex,circle,fill=black]{};
\node [below] at (0,0) {$v_1$};


\draw (v1) edge node[below,sloped,text=red,font=\tiny]{} (vn);
\draw (v1) edge node[below,sloped,text=red,font=\tiny]{} (vn-1);
\draw (v1) edge node[below,sloped,text=red,font=\tiny]{} (vn-2);
\draw (v1) edge node[above,sloped,text=red,font=\tiny]{} (vn-3);
\draw (v1) edge node[above,sloped,text=red,font=\tiny]{} (vn-4);
\draw (v1) edge node[above,sloped,text=red,font=\tiny]{} (v2);
\draw (v1) edge node[above,sloped,text=red,font=\tiny]{} (v3);
\draw (vn) edge node[above left=-0.1,anchor=south east,text=red,font=\tiny]{$d_n-2-\alpha$} (vn-1);
\draw (vn-1) edge node[above right=-0.1,anchor=south west,text=red,font=\tiny]{$d_{n-1}-d_n+1-\beta$} (vn-2);
\draw (vn-4) edge node[right,text=red,font=\tiny]{$d_{n-3}-1$} (vn-3);
\draw (v2) edge node[left,text=red,font=\tiny]{$d_3-1$} (v3);
\end{tikzpicture}
\caption{$n$ even and $k=1$}
\label{fig:bigD2}
\end{subfigure}
\begin{subfigure}{\textwidth}
\centering
\begin{tikzpicture}[scale=2.5]


\draw ($({-cos(66)},{sin(66)})$) node(vn)[inner sep=0.3ex,circle,fill=black]{};
\node [above left] at ($({-cos(66)},{sin(66)})$) {$v_n$};
\draw ($(0,1)$) node(vn-1)[inner sep=0.3ex,circle,fill=black]{};
\node [above] at ($(0,1)$) {$v_{n-1}$};
\draw ($({cos(66)},{sin(66)})$) node(vn-2)[inner sep=0.3ex,circle,fill=black]{};
\node [above right] at ($({cos(66)},{sin(66)})$) {$v_{n-2}$};
\draw ($({cos(42)},{sin(42)})$) node(vn-3)[inner sep=0.3ex,circle,fill=black]{};
\node [right] at ($({cos(42)},{sin(42)})$) {$v_{n-3}$};
\draw ($({cos(18)},{sin(18)})$) node(vn-4)[inner sep=0.3ex,circle,fill=black]{};
\node [right] at ($({cos(18)},{sin(18)})$) {$v_{n-4}$};
\draw ($({-cos(42)},{sin(42)})$) node(v2)[inner sep=0.3ex,circle,fill=black]{};
\node [left] at ($({-cos(42)},{sin(42)})$) {$v_2$};
\draw ($({-cos(18)},{sin(18)})$) node(v3)[inner sep=0.3ex,circle,fill=black]{};
\node [left] at ($({-cos(18)},{sin(18)})$) {$v_3$};
\draw ($({-cos(30)},{-sin(30)})$) node(vn-2k-2)[inner sep=0.3ex,circle,fill=black]{};
\node [left] at ($({-cos(30)},{-sin(30)})$) {$v_{n-2k-2}$};
\draw ($({-cos(54)},{-sin(54)})$) node(vn-2k-1)[inner sep=0.3ex,circle,fill=black]{};
\node [left] at ($({-cos(54)},{-sin(54)})$) {$v_{n-2k-1}$};
\draw ($({cos(30)},{-sin(30)})$) node(vn-2k+3)[inner sep=0.3ex,circle,fill=black]{};
\node [right] at ($({cos(30)},{-sin(30)})$) {$v_{n-2k+3}$};
\draw ($({cos(54)},{-sin(54)})$) node(vn-2k+2)[inner sep=0.3ex,circle,fill=black]{};
\node [right] at ($({cos(54)},{-sin(54)})$) {$v_{n-2k+2}$};
\draw ($({cos(78)},{-sin(78)})$) node(vn-2k+1)[inner sep=0.3ex,circle,fill=black]{};
\node [right] at ($({cos(78)},{-sin(78)})$) {$v_{n-2k+1}$};
\draw ($({-cos(78)},{-sin(78)})$) node(vn-2k)[inner sep=0.3ex,circle,fill=black]{};
\node [left] at ($({-cos(78)},{-sin(78)})$) {$v_{n-2k}$};
\draw (0,0) node(v1)[inner sep=0.3ex,circle,fill=black]{};
\node [right] at (0.1,0) {$v_1$};

\node at (0.75,0) {$\vdots$};
\node at (-0.75,0) {$\vdots$};


\draw (v1) edge node[above,sloped,text=red,font=\tiny]{} (v2);
\draw (v1) edge node[above,sloped,text=red,font=\tiny]{} (v3);
\draw (v1) edge node[above,sloped,text=red,font=\tiny]{} (vn-2k-2);
\draw (v1) edge node[above,sloped,text=red,font=\tiny]{} (vn-2k-1);
\draw (v1) edge node[above,sloped,text=red,font=\tiny]{} (vn-2k);
\draw (v1) edge node[above,sloped,text=red,font=\tiny]{} (vn-2k+1);
\draw (v1) edge node[above,sloped,text=red,font=\tiny]{} (vn-2k+2);
\draw (v1) edge node[above,sloped,text=red,font=\tiny]{} (vn-2k+3);
\draw (v1) edge node[above,sloped,text=red,font=\tiny]{} (vn-4);
\draw (v1) edge node[above,sloped,text=red,font=\tiny]{} (vn-3);
\draw (v1) edge node[above,sloped,text=red,font=\tiny]{} (vn-2);
\draw (v1) edge node[above,sloped,text=red,font=\tiny]{} (vn-1);
\draw (v1) edge node[below,sloped,text=red,font=\tiny]{} (vn);
\draw (vn-1) edge node[above left,text=red,font=\tiny]{$1$} (vn);
\draw (vn-1) edge node[above right,text=red,font=\tiny]{$1$} (vn-2);
\draw (vn-3) edge node[right,text=red,font=\tiny]{$1$} (vn-4);
\draw (vn-2k+3) edge node[right=0.1,text=red,font=\tiny]{$1$} (vn-2k+2);
\draw (vn-2k+1) edge node[below=0.1,text=red,font=\tiny]{$d_{n-2k+1}-1-\delta$} (vn-2k);
\draw (vn-2k-1) edge node[left=0.1,text=red,font=\tiny]{$d_{n-2k-1}-1$} (vn-2k-2);
\draw (v2) edge node[left,text=red,font=\tiny]{$d_3-1$} (v3);
\end{tikzpicture}
\caption{$n$ even and $k>1$}
\label{fig:bigD3}
\end{subfigure}
\caption{The fan-shaped constructions in Lemma \ref{bigD}. For simplicity only multiplicities of edges not containing $v_1$ are labelled. Multiplicities of edges containing $v_1$ are included in the proof and can be deduced using $\deg(v_i)=d_i$ for all $i\in[n]$.}
\label{fig:bigD}
\end{figure}
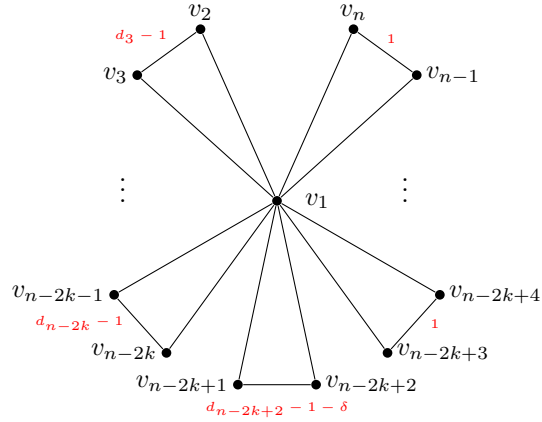
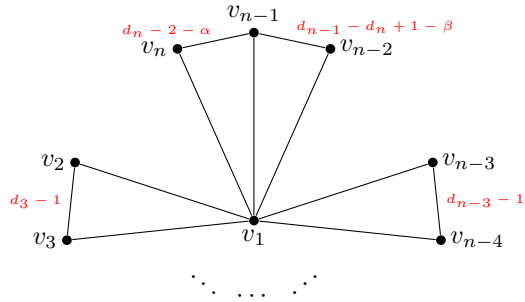
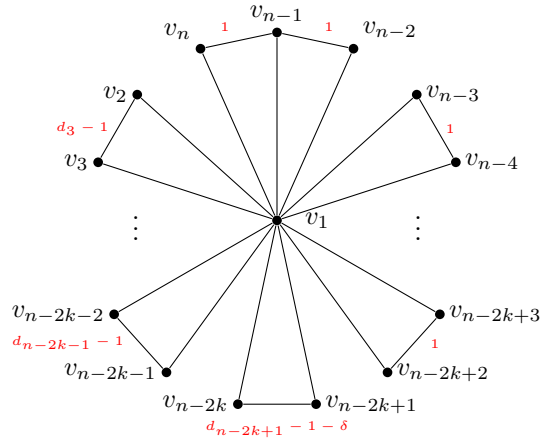
\begin{proof}
We separate into two cases depending on the parity of $n$. 

If $n$ is odd, then $D\geq n-1$ as $D$ is even. Let $\overline{d}_{n+2-2i}=d_{n+2-2i}-2$ for each $1\leq i\leq\frac{n-1}2$. Using $d_1\leq\sum_{i=2}^n(d_i-1)$, we have
\begin{align*}
\frac12\left(D-(n-1)\right)&=\frac12\left(\sum_{i=1}^n(-1)^{i-1}d_i-(n-1)\right)\\
&\leq\frac12\left(\sum_{i=2}^nd_i+\sum_{i=2}^n(-1)^{i-1}d_i-2(n-1)\right)\\
&=\sum_{i=1}^{\frac{n-1}2}d_{n+2-2i}-(n-1)=\sum_{i=1}^{\frac{n-1}2}\overline{d}_{n+2-2i}.
\end{align*}
Hence, there exists an index $1\leq k\leq\frac{n-1}2$ such that $\sum_{i=1}^{k-1}\overline{d}_{n+2-2i}\leq\frac12(D-(n-1))\leq\sum_{i=1}^k\overline{d}_{n+2-2i}$. Let $\delta=\frac12(D-(n-1))-\sum_{i=1}^{k-1}\overline{d}_{n+2-2i}$, so $0\leq\delta\leq\overline{d}_{n+2-2k}=d_{n+2-2k}-2$. Consider the multigraph $G$ on $n$ vertices $v_1,\cdots,v_n$ whose edge multiplicities are given as follows (see also Figure \ref{fig:bigD1}). For each $i\in[k-1]$, let $m(v_1,v_{n-2i+2})=d_{n-2i+2}-1$, $m(v_1,v_{n-2i+1})=d_{n-2i+1}-1$, and $m(v_{n-2i+1},v_{n-2i+2})=1$. For each $k+1\leq i\leq\frac{n-1}2$, let $m(v_1,v_{n-2i+2})=1$, $m(v_1,v_{n-2i+1})=1+d_{n-2i+1}-d_{n-2i+2}$, and $m(v_{n-2i+1},v_{n-2i+2})=d_{n-2i+2}-1$. Finally, let $m(v_1,v_{n-2k+2})=1+\delta$, $m(v_1,v_{n-2k+1})=1+\delta+d_{n-2k+1}-d_{n-2k+2}$, and $m(v_{n-2k+1},v_{n-2k+2})=d_{n-2k+2}-1-\delta$. Note that every edge mentioned so far has multiplicity at least 1. Let every other edges in $G$ have multiplicity 0. It follows that $\deg(v_i)=d_i$ for all $2\leq i\leq n$ and
\begin{align*}
\deg(v_1)&=\sum_{i=2}^nm(v_1,v_i)\\
&=\sum_{i=1}^{k-1}(d_{n-2i+1}+d_{n-2i+2}-2)+\sum_{i=k}^{\frac{n-1}2}(d_{n-2i+1}-d_{n-2i+2}+2)+2\delta\\
&=d_1-D+2\sum_{i=1}^{k-1}\overline{d}_{n-2i+2}+(n-1)+2\delta=d_1,
\end{align*}
where the last equality follows from the definition of $\delta$. Hence, $G$ is a multigraph with degree sequence $(d_1,\cdots,d_n)$. Moreover, as an edge in $G$ has positive multiplicity if and only if it is of the form $v_1v_i$ for $2\leq i\leq n$, or of the form $v_{2i}v_{2i+1}$ for some $1\leq i\leq \frac{n-1}2$, we see that $G$ is triangular, completing the proof of the odd case.

If $n$ is even, then $d_1-\sum_{i=2}^nd_i$ is also even, and thus $d_1\leq\sum_{i=2}^nd_i-n$. Let $\overline{d}_{n-1}=d_{n-1}-3$, and for each $2\leq i\leq\frac{n-2}2$, let $\overline{d}_{n+1-2i}=d_{n+1-2i}-2$. It follows that
\begin{align*}
\frac12\left(D-(n-2)\right)&=\frac12\left(\sum_{i=1}^n(-1)^{i-1}d_i-(n-2)\right)\\
&\leq\frac12\left(\sum_{i=2}^nd_i+\sum_{i=2}^n(-1)^{i-1}d_i-(2n-2)\right)\\
&=\sum_{i=1}^{\frac{n-2}2}d_{n+1-2i}-(n-1)=\sum_{i=1}^{\frac{n-2}2}\overline{d}_{n+1-2i}.
\end{align*}
Hence, there exists an index $1\leq k\leq \frac{n-2}2$ such that $\sum_{i=1}^{k-1}\overline{d}_{n+1-2i}\leq\frac12(D-(n-2))\leq\sum_{i=1}^k\overline{d}_{n+1-2i}$. Let $\delta=\frac12(D-(n-2))-\sum_{i=1}^{k-1}\overline{d}_{n+1-2i}$, so $0\leq\delta\leq\overline{d}_{n+1-2k}$.

If $k=1$, let $\alpha,\beta$ be any non-negative integers satisfying $\alpha\leq d_n-3$, $\beta\leq d_{n-1}-d_n$ and $\alpha+\beta=\delta$. Such $\alpha,\beta$ exists as $d_n-3+d_{n-1}-d_n=\overline{d}_{n-1}\geq\delta$. Consider the multigraph $G$ on $n$ vertices $v_1,\cdots,v_n$ whose edge multiplicities are given as follows (see also Figure \ref{fig:bigD2}). Let $m(v_1,v_n)=2+\alpha$, $m(v_1,v_{n-1})=1+\alpha+\beta$, $m(v_1,v_{n-2})=d_{n-2}-d_{n-1}+d_n-1+\beta$, $m(v_n,v_{n-1})=d_n-2-\alpha$ and $m(v_{n-1},v_{n-2})=d_{n-1}-d_n+1-\beta$. For each $2\leq i\leq \frac{n-2}2$, let $m(v_1,v_{n-2i+1})=1$, $m(v_1,v_{n-2i})=1+d_{n-2i}-d_{n-2i+1}$, and $m(v_{n-2i},v_{n-2i+1})=d_{n-2i+1}-1$. Note that every edge mentioned so far has multiplicity at least 1. Let every other edges in $G$ have multiplicity 0. Then, $\deg(v_i)=d_i$ for all $2\leq i\leq n$ and
\begin{align*}
\deg(v_1)&=2+2\alpha+2\beta+d_{n-2}-d_{n-1}+d_n+\sum_{i=2}^{\frac{n-2}2}(2+d_{n-2i}-d_{n-2i+1})\\
&=2+2\delta+d_1-D+(n-4)=d_1.
\end{align*}

If $k>1$, consider the multigraph $G$ on $n$ vertices $v_1,\cdots,v_n$ whose edge multiplicities are given as follows (see also Figure \ref{fig:bigD3}). Let $m(v_1,v_n)=d_n-1$, $m(v_1,v_{n-1})=d_{n-1}-2$, $m(v_1,v_{n-2})=d_{n-2}-1$, and $m(v_n,v_{n-1})=m(v_{n-1},v_{n-2})=1$. For each $2\leq i\leq k-1$, $m(v_1,v_{n-2i+1})=d_{n-2i+1}-1$, $m(v_1,v_{n-2i})=d_{n-2i}-1$, and $m(v_{n-2i},v_{n-2i+1})=1$. For each $k+1\leq i\leq\frac{n-2}2$, $m(v_1,v_{n-2i+1})=1$, $m(v_1,v_{n-2i})=1+d_{n-2i}-d_{n-2i+1}$, and $m(v_{n-2i},v_{n-2i+1})=d_{n-2i+1}-1$. Finally, let $m(v_1,v_{n-2k+1})=1+\delta$, $m(v_1,v_{n-2k})=1+\delta+d_{n-2k}-d_{n-2k+1}$, and $m(v_{n-2k},v_{n-2k+1})=d_{n-2k+1}-1-\delta$. Note that from assumptions, every edge mentioned so far has multiplicity at least 1. Let every other edges in $G$ have multiplicity 0. Then $\deg(v_i)=d_i$ for all $2\leq i\leq n$ and
\begin{align*}
\deg(v_1)&=d_n+d_{n-1}+d_{n-2}-4+\sum_{i=2}^{k-1}(d_{n-2i+1}+d_{n-2i}-2)+\sum_{i=k}^{\frac{n-2}2}(d_{n-2i}-d_{n-2i+1}+2)+2\delta\\
&=d_1-D+2\sum_{i=1}^{k-1}\overline{d}_{n-2i+1}+(n-2)+2\delta=d_1.
\end{align*}
Therefore, regardless of whether or not $k=1$, $G$ is a multigraph with degree sequence $(d_1,\cdots,d_n)$. Moreover, we have that an edge in $G$ has positive multiplicity if and only if it is of the form $v_1v_i$ for $2\leq i\leq n$, or of the form $v_{2i}v_{2i+1}$ for some $1\leq i\leq \frac{n-2}2$, or it is $v_{n-1}v_n$. Hence, $G$ is triangular, proving the even case.
\end{proof}

\begin{lemma}\label{smallD}
Let $n\geq 5$ and let $(d_1,\cdots,d_n)$ be a sequence of non-negative integers satisfying
\begin{itemize}
\item $d_2\geq\cdots\geq d_n\geq4$,
\item $D=\sum_{i=1}^n(-1)^{i-1}d_i$ is equal to $4$ if $n$ is odd, and one of $0,2,4$ if $n$ is even,
\end{itemize}
then there exists a triangular multigraph $G$ with degree sequence $(d_1,\cdots,d_n)$.    
\end{lemma}
\begin{figure}
\centering
\begin{subfigure}{\textwidth}
\centering
\begin{tikzpicture}[scale=1.5]
\draw ($(0,1)$) -- ($({cos(54)},{sin(54)})$) -- ($({cos(18)},{sin(18)})$) -- ($({cos(18)},-{sin(18)})$) -- ($({cos(54)},-{sin(54)})$);
\draw ($(0,1)$) -- ($(-{cos(54)},{sin(54)})$) -- ($(-{cos(18)},{sin(18)})$) -- ($(-{cos(18)},-{sin(18)})$) -- ($(-{cos(54)},-{sin(54)})$);
\draw ($(-{cos(54)},-{sin(54)})$) -- ($(-{cos(18)},{sin(18)})$) -- (0,1) -- ($({cos(18)},{sin(18)})$) -- ($({cos(54)},-{sin(54)})$);
\draw ($(-{cos(18)},-{sin(18)})$) -- ($(-{cos(54)},{sin(54)})$) -- ($({cos(54)},{sin(54)})$) -- ($({cos(18)},-{sin(18)})$);
\draw ($(0,1)$) node[inner sep=0.3ex,circle,fill=black]{};
\node [above] at ($(0,1)$) {$v_1$};
\draw ($({cos(54)},{sin(54)})$) node[inner sep=0.3ex,circle,fill=black]{};
\node [above right] at ($({cos(54)},{sin(54)})$) {$v_n$};
\draw ($({cos(18)},{sin(18)})$) node[inner sep=0.3ex,circle,fill=black]{};
\node [right] at ($({cos(18)},{sin(18)})$) {$v_{n-1}$};
\draw ($({cos(18)},-{sin(18)})$) node[inner sep=0.3ex,circle,fill=black]{};
\node [right] at ($({cos(18)},-{sin(18)})$) {$v_{n-2}$};
\draw ($({cos(54)},-{sin(54)})$) node[inner sep=0.3ex,circle,fill=black]{};
\node [below right] at ($({cos(54)},-{sin(54)})$) {$v_{n-3}$};
\draw ($(-{cos(54)},-{sin(54)})$) node[inner sep=0.3ex,circle,fill=black]{};
\node [below left] at ($(-{cos(54)},-{sin(54)})$) {$v_5$};
\draw ($(-{cos(18)},-{sin(18)})$) node[inner sep=0.3ex,circle,fill=black]{};
\node [left] at ($(-{cos(18)},-{sin(18)})$) {$v_4$};
\draw ($(-{cos(18)},{sin(18)})$) node[inner sep=0.3ex,circle,fill=black]{};
\node [left] at ($(-{cos(18)},{sin(18)})$) {$v_3$};
\draw ($(-{cos(54)},{sin(54)})$) node[inner sep=0.3ex,circle,fill=black]{};
\node [above left] at ($(-{cos(54)},{sin(54)})$) {$v_2$};
\draw ($(0,-2/3)$) node {$\cdots$};

\node [anchor=west, red, font=\tiny] at ($({0.5*cos(54)+0.5*cos(18)},{0.55*sin(54)+0.55*sin(18)})$) {$1+D_n$};
\node [red, right, font=\tiny] at ($({cos(18)},0)$) {$1+D_{n-1}$};
\node [anchor=west, red, font=\tiny] at ($({0.5*cos(54)+0.5*cos(18)},{-0.55*sin(54)-0.55*sin(18)})$) {$1+D_{n-2}$};
\node [anchor=east, red, font=\tiny] at ($({-0.5*cos(54)-0.5*cos(18)},{-0.55*sin(54)-0.55*sin(18)})$) {$1+D_5$};
\node [red, left, font=\tiny] at ($({-cos(18)},0)$) {$1+D_4$};
\node [anchor=east, red, font=\tiny] at ($({-0.5*cos(54)-0.5*cos(18)},{0.55*sin(54)+0.55*sin(18)})$) {$1+D_3$};
\node [anchor=east, red, font=\tiny] at ($({-0.15*cos(54)},{0.6*sin(54)+0.6})$) {$1+D_2$};
\end{tikzpicture}
\caption{$n$ odd, or $n$ even and $D=0$}
\label{fig:smallD1}
\end{subfigure}
\begin{subfigure}{0.4\textwidth}
\centering
\begin{tikzpicture}[scale=1.5]
\draw ($(0,1)$) -- ($({cos(54)},{sin(54)})$) -- ($({cos(18)},{sin(18)})$) -- ($({cos(18)},-{sin(18)})$) -- ($({cos(54)},-{sin(54)})$);
\draw ($(0,1)$) -- ($(-{cos(54)},{sin(54)})$) -- ($(-{cos(18)},{sin(18)})$) -- ($(-{cos(18)},-{sin(18)})$) -- ($(-{cos(54)},-{sin(54)})$);
\draw ($(-{cos(54)},-{sin(54)})$) -- ($(-{cos(18)},{sin(18)})$) -- (0,1) -- ($({cos(18)},{sin(18)})$) -- ($({cos(54)},-{sin(54)})$);
\draw ($(-{cos(18)},-{sin(18)})$) -- ($(-{cos(54)},{sin(54)})$);
\draw ($({cos(54)},{sin(54)})$) -- ($({cos(18)},-{sin(18)})$);
\draw ($(0,1)$) node[inner sep=0.3ex,circle,fill=black]{};
\node [above] at ($(0,1)$) {$v_1$};
\draw ($({cos(54)},{sin(54)})$) node[inner sep=0.3ex,circle,fill=black]{};
\node [above right] at ($({cos(54)},{sin(54)})$) {$v_n$};
\draw ($({cos(18)},{sin(18)})$) node[inner sep=0.3ex,circle,fill=black]{};
\node [right] at ($({cos(18)},{sin(18)})$) {$v_{n-1}$};
\draw ($({cos(18)},-{sin(18)})$) node[inner sep=0.3ex,circle,fill=black]{};
\node [right] at ($({cos(18)},-{sin(18)})$) {$v_{n-2}$};
\draw ($({cos(54)},-{sin(54)})$) node[inner sep=0.3ex,circle,fill=black]{};
\node [below right] at ($({cos(54)},-{sin(54)})$) {$v_{n-3}$};
\draw ($(-{cos(54)},-{sin(54)})$) node[inner sep=0.3ex,circle,fill=black]{};
\node [below left] at ($(-{cos(54)},-{sin(54)})$) {$v_5$};
\draw ($(-{cos(18)},-{sin(18)})$) node[inner sep=0.3ex,circle,fill=black]{};
\node [left] at ($(-{cos(18)},-{sin(18)})$) {$v_4$};
\draw ($(-{cos(18)},{sin(18)})$) node[inner sep=0.3ex,circle,fill=black]{};
\node [left] at ($(-{cos(18)},{sin(18)})$) {$v_3$};
\draw ($(-{cos(54)},{sin(54)})$) node[inner sep=0.3ex,circle,fill=black]{};
\node [above left] at ($(-{cos(54)},{sin(54)})$) {$v_2$};
\draw ($(0,-2/3)$) node {$\cdots$};

\node [anchor=west, red, font=\tiny] at ($({0.5*cos(54)+0.5*cos(18)},{0.55*sin(54)+0.55*sin(18)})$) {$1+D_n$};
\node [red, right, font=\tiny] at ($({cos(18)},0)$) {$1+D_{n-1}$};
\node [anchor=west, red, font=\tiny] at ($({0.5*cos(54)+0.5*cos(18)},{-0.55*sin(54)-0.55*sin(18)})$) {$1+D_{n-2}$};
\node [anchor=east, red, font=\tiny] at ($({-0.5*cos(54)-0.5*cos(18)},{-0.55*sin(54)-0.55*sin(18)})$) {$1+D_5$};
\node [red, left, font=\tiny] at ($({-cos(18)},0)$) {$1+D_4$};
\node [anchor=east, red, font=\tiny] at ($({-0.5*cos(54)-0.5*cos(18)},{0.55*sin(54)+0.55*sin(18)})$) {$1+D_3$};
\node [anchor=east, red, font=\tiny] at ($({-0.15*cos(54)},{0.6*sin(54)+0.6})$) {$2+D_2$};
\node [red, font=\tiny] at ($({0.6*cos(54)},{0.6*sin(54)+0.6})$) {$2$};
\end{tikzpicture}
\caption{$n$ even and $D=2$}
\label{fig:smallD2}
\end{subfigure}
\begin{subfigure}{0.4\textwidth}
\centering
\begin{tikzpicture}[scale=1.5]
\draw ($(0,1)$) -- ($({cos(54)},{sin(54)})$) -- ($({cos(18)},{sin(18)})$) -- ($({cos(18)},-{sin(18)})$) -- ($({cos(54)},-{sin(54)})$);
\draw ($(0,1)$) -- ($(-{cos(54)},{sin(54)})$) -- ($(-{cos(18)},{sin(18)})$) -- ($(-{cos(18)},-{sin(18)})$) -- ($(-{cos(54)},-{sin(54)})$);
\draw ($(-{cos(54)},-{sin(54)})$) -- ($(-{cos(18)},{sin(18)})$) -- (0,1) -- ($({cos(18)},{sin(18)})$) -- ($({cos(54)},-{sin(54)})$);
\draw ($({cos(54)},{sin(54)})$) -- ($({cos(18)},-{sin(18)})$);
\draw (0,1) -- ($(-{cos(18)},-{sin(18)})$);
\draw ($(0,1)$) node[inner sep=0.3ex,circle,fill=black]{};
\node [above] at ($(0,1)$) {$v_1$};
\draw ($({cos(54)},{sin(54)})$) node[inner sep=0.3ex,circle,fill=black]{};
\node [above right] at ($({cos(54)},{sin(54)})$) {$v_n$};
\draw ($({cos(18)},{sin(18)})$) node[inner sep=0.3ex,circle,fill=black]{};
\node [right] at ($({cos(18)},{sin(18)})$) {$v_{n-1}$};
\draw ($({cos(18)},-{sin(18)})$) node[inner sep=0.3ex,circle,fill=black]{};
\node [right] at ($({cos(18)},-{sin(18)})$) {$v_{n-2}$};
\draw ($({cos(54)},-{sin(54)})$) node[inner sep=0.3ex,circle,fill=black]{};
\node [below right] at ($({cos(54)},-{sin(54)})$) {$v_{n-3}$};
\draw ($(-{cos(54)},-{sin(54)})$) node[inner sep=0.3ex,circle,fill=black]{};
\node [below left] at ($(-{cos(54)},-{sin(54)})$) {$v_5$};
\draw ($(-{cos(18)},-{sin(18)})$) node[inner sep=0.3ex,circle,fill=black]{};
\node [left] at ($(-{cos(18)},-{sin(18)})$) {$v_4$};
\draw ($(-{cos(18)},{sin(18)})$) node[inner sep=0.3ex,circle,fill=black]{};
\node [left] at ($(-{cos(18)},{sin(18)})$) {$v_3$};
\draw ($(-{cos(54)},{sin(54)})$) node[inner sep=0.3ex,circle,fill=black]{};
\node [above left] at ($(-{cos(54)},{sin(54)})$) {$v_2$};
\draw ($(0,-2/3)$) node {$\cdots$};

\node [anchor=west, red, font=\tiny] at ($({0.5*cos(54)+0.5*cos(18)},{0.55*sin(54)+0.55*sin(18)})$) {$1+D_n$};
\node [red, right, font=\tiny] at ($({cos(18)},0)$) {$1+D_{n-1}$};
\node [anchor=west, red, font=\tiny] at ($({0.5*cos(54)+0.5*cos(18)},{-0.55*sin(54)-0.55*sin(18)})$) {$1+D_{n-2}$};
\node [anchor=east, red, font=\tiny] at ($({-0.5*cos(54)-0.5*cos(18)},{-0.55*sin(54)-0.55*sin(18)})$) {$1+D_5$};
\node [red, left, font=\tiny] at ($({-cos(18)},0)$) {$1+D_4$};
\node [anchor=east, red, font=\tiny] at ($({-0.5*cos(54)-0.5*cos(18)},{0.55*sin(54)+0.55*sin(18)})$) {$1+D_3$};
\node [anchor=east, red, font=\tiny] at ($({-0.15*cos(54)},{0.6*sin(54)+0.6})$) {$3+D_2$};
\node [red, font=\tiny] at ($({0.6*cos(54)},{0.6*sin(54)+0.6})$) {$2$};
\end{tikzpicture}
\caption{$n$ even and $D=4$}
\label{fig:smallD3}
\end{subfigure}
\caption{The constructions in Lemma \ref{smallD}. Every edge here with no labelled multiplicity has multiplicity 1.}
\label{fig:smallD}
\end{figure}
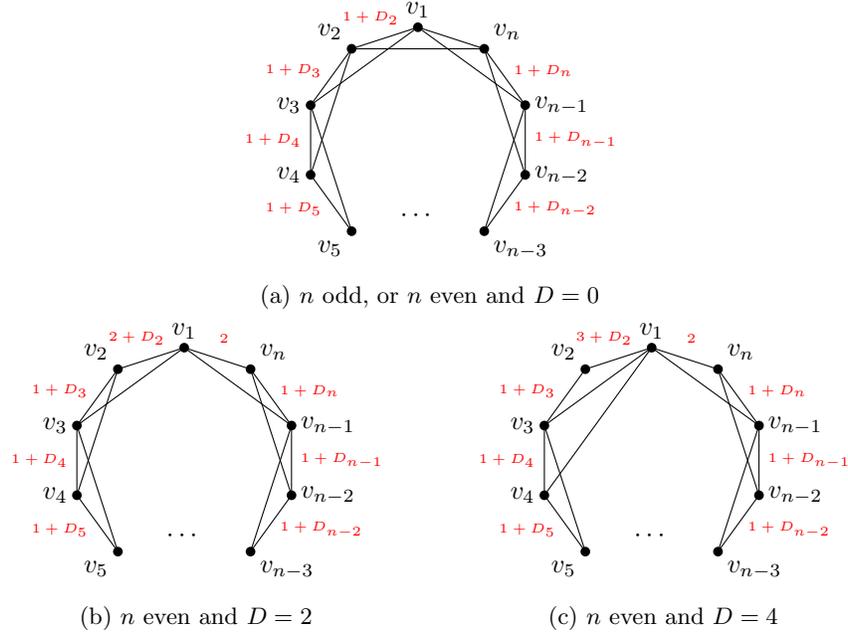
\begin{proof}
For each $i\in[n]$, let $d'_i=d_i-4$. For each $2\leq i\leq n$, let $D_i=\sum_{j=i}^n(-1)^{j-i}d'_j$, and note that $D_i\geq 0$.

If $n$ is odd, then by assumption $D=4$. Consider the multigraph $G$ on $n$ vertices $v_1,\cdots,v_n$ with edge multiplicities $m(v_i,v_j)$ given as follows (see also Figure \ref{fig:smallD1}), where we use addition mod $n$ in the indices. For each $i\in[n-1]$, let $m(v_i,v_{i+1})=1+D_{i+1}$, and let $m(v_n,v_1)=1$. For each $i\in[n]$, let $m(v_i,v_{i+2})=1$. Let every other edge in $G$ have multiplicity 0. Note that for each $2\leq i\leq n-1$, \begin{align*}
\deg(v_i)&=m(v_i,v_{i-2})+m(v_i,v_{i-1})+m(v_i,v_{i+1})+m(v_i,v_{i+2})\\
&=1+\left(1+D_i\right)+\left(1+D_{i+1}\right)+1\\
&=4+d'_i=d_i,
\end{align*}
and similarly
\begin{align*}
\deg(v_1)&=4+D_2=4+\sum_{j=2}^n(-1)^{j}d'_j=4+\sum_{j=2}^n(-1)^jd_j=4+d_1-D=d_1,\\
\deg(v_n)&=4+D_n=4+d'_n=d_n.
\end{align*}
Hence, $G$ is a multigraph with degree sequence $(d_1,\cdots,d_n)$. Furthermore, since an edge in $G$ has positive multiplicity if and only if it connects vertices whose indices have difference 1 or 2 mod $n$, we see that $G$ is triangular, completing the proof when $n$ is odd. 

If $n$ is even, then by assumption $D$ could be 0, 2 or 4. Let $G$ be the multigraph with the same definition as in the case when $n$ is odd (see also Figure \ref{fig:smallD1}). The same calculations show $\deg(v_i)=d_i$ for all $2\leq i\leq n$, while
\begin{align*}
\deg(v_1)&=4+D_2=4+(d'_2-d'_3)+\cdots+(d'_{n-2}-d'_{n-1})+d'_n\\
&=4+(d_2-d_3)+\cdots+(d_{n-2}-d_{n-1})+(d_n-4)=4+d_1-D-4=d_1-D.
\end{align*}
Hence, if $D=0$, then $G$ is a triangular multigraph with degree sequence $(d_1,\cdots,d_n)$, as required.

If $D=2$, let $G'$ be the multigraph obtained from $G$ by increasing $m(v_1,v_2)$ and $m(v_1,v_n)$ by 1, and decreasing $m(v_2,v_n)$ by 1 to 0 (see also Figure \ref{fig:smallD2}). Note that $G'$ is a multigraph with degree sequence $(d_1,\cdots,d_n)$. As edge $v_1v_2$ is in triangle $v_1v_2v_3$ and edge $v_1v_n$ is in triangle $v_1v_nv_{n-1}$, $G'$ is still triangular. 

If $D=4$, let $G''$ be the multigraph obtained from $G$ by increasing $m(v_1,v_2)$ by 2 and $m(v_1,v_n)$ by 1, increasing $m(v_1,v_4)$ from 0 to 1, and decreasing both $m(v_2,v_n)$ and $m(v_2,v_4)$ by 1 to 0 (see also Figure \ref{fig:smallD3}). Note that $G''$ is a multigraph with degree sequence $(d_1,\cdots,d_n)$. As edges $v_1v_2$ and $v_2v_3$ are in triangle $v_1v_2v_3$, edge $v_3v_4$ is in triangle $v_3v_4v_5$, edge $v_1v_4$ is in triangle $v_1v_3v_4$, and edge $v_1v_n$ is in triangle $v_1v_nv_{n-1}$, $G''$ is still triangular. This completes the proof when $n$ is even. 
\end{proof}

As a final preparation, we deal with the $n=3$ and $n=4$ cases of Theorem \ref{main} in the following lemma. 
\begin{lemma}\label{n=3,4}
If $n=3$ or $n=4$ and $(d_1,\cdots,d_n)$ is a sequence of positive integers satisfying (\ref{decreasing})-(\ref{d1atmost}), then there exists a triangular multigraph $G$ with degree sequence $(d_1,\cdots,d_n)$.
\end{lemma}
\begin{proof}
If $n=3$, then $D=d_1-d_2+d_3\geq d_3\geq 4\geq 3-2$. Hence, we may apply Lemma \ref{bigD} to find such a triangular multigraph $G$.  

If $n=4$ and $D=d_1-d_2+d_3-d_4\geq 2=4-2$, then we may again apply Lemma \ref{bigD} to find such a triangular multigraph $G$. Since $D$ is even, the only remaining case is $D=0$, which can only happen if $d_1=d_2$ and $d_3=d_4$. Consider the multigraph $G$ on $v_1,v_2,v_3,v_4$ given by $m(v_1,v_2)=d_1-2$, $m(v_3,v_4)=d_3-2$, and $m(v_1,v_3)=m(v_1,v_4)=m(v_2,v_3)=m(v_2,v_4)=1$. Then $G$ has degree sequence $(d_1,d_2,d_3,d_4)$ and is triangular, completing the proof. 
\end{proof}

We are now ready to prove Theorem \ref{main}.
\begin{figure}
\centering
\begin{subfigure}{0.4\textwidth}
\centering
\begin{tikzpicture}[scale=2]
\draw ($({-cos(60)},{sin(60)})$) node(vn-1)[inner sep=0.3ex,circle,fill=black]{};
\node [above] at ($({-cos(60)},{sin(60)})$) {$v_{n-1}$};
\draw ($({-cos(30)},{sin(30)})$) node(vn)[inner sep=0.3ex,circle,fill=black]{};
\node [above left] at ($({-cos(30)},{sin(30)})$) {$v_n$};
\draw ($({cos(30)},{sin(30)})$) node(vn-2k+1)[inner sep=0.3ex,circle,fill=black]{};
\node [above right] at ($({cos(30)},{sin(30)})$) {$v_{n-2k+1}$};
\draw ($({cos(60)},{sin(60)})$) node(vn-2k+2)[inner sep=0.3ex,circle,fill=black]{};
\node [above] at ($({cos(60)},{sin(60)})$) {$v_{n-2k+2}$};


\draw ($({-cos(45)},{sin(45)-1})$) node(v2)[inner sep=0.3ex,circle,fill=black]{};
\node [above left] at ($({-cos(45)},{sin(45)-1})$) {$v_2$};
\draw (-1,-1) node(v3)[inner sep=0.3ex,circle,fill=black]{};
\node [left] at (-1,-1) {$v_3$};
\draw ($({-cos(45)},{-sin(45)-1})$) node(v4)[inner sep=0.3ex,circle,fill=black]{};
\node [below left] at ($({-cos(45)},{-1-sin(45)})$) {$v_4$};
\draw ($({cos(45)},{-1-sin(45)})$) node(vn-2k-2)[inner sep=0.3ex,circle,fill=black]{};
\node [below right] at ($({cos(45)},{-1-sin(45)})$) {$v_{n-2k-2}$};
\draw (1,-1) node(vn-2k-1)[inner sep=0.3ex,circle,fill=black]{};
\node [right] at (1,-1) {$v_{n-2k-1}$};
\draw ($({cos(45)},{-1+sin(45)})$) node(vn-2k)[inner sep=0.3ex,circle,fill=black]{};
\node [above right] at ($({cos(45)},{-1+sin(45)})$) {$v_{n-2k}$};
\draw (0,0) node(v1)[inner sep=0.3ex,circle,fill=black]{};
\node [above] at (0,0.1) {$v_1$};
\node at (0,0.75) {$\cdots$};
\node at (0,-1.75) {$\cdots$};

\draw (v1) edge (vn);
\draw (v1) edge (vn-1);
\draw (v1) edge (v2);
\draw (v1) edge (vn-2k+2);
\draw (v1) edge (vn-2k+1);
\draw (v1) edge (vn-2k);
\draw (vn-1) edge (vn);
\draw (vn-2k+2) edge (vn-2k+1);
\draw (v2) edge (v3);
\draw (v3) edge (v4);
\draw (vn-2k-1) edge (vn-2k-2);
\draw (vn-2k) edge (vn-2k-1);
\draw (v1) edge (v3);
\draw (v2) edge (v4);
\draw (vn-2k) edge (vn-2k-2);
\draw (v1) edge (vn-2k-1);
\draw (v2) edge (vn-2k);

\end{tikzpicture}
\caption{$n$ odd}
\end{subfigure}
\begin{subfigure}{0.4\textwidth}
\centering
\begin{tikzpicture}[scale=2]
\draw ($({-cos(33.75)},{sin(33.75)})$) node(vn-1)[inner sep=0.3ex,circle,fill=black]{};
\node [above left] at ($({-cos(33.75)},{sin(33.75)})$) {$v_{n-1}$};
\draw ($({-cos(11.25)},{sin(11.25)})$) node(vn)[inner sep=0.3ex,circle,fill=black]{};
\node [left] at ($({-cos(11.25)},{sin(11.25)})$) {$v_n$};
\node [above] at ($({-cos(56.25)},{sin(56.25)})$) {$v_{n-2}$};
\draw ($({-cos(56.25)},{sin(56.25)})$) node(vn-2)[inner sep=0.3ex,circle,fill=black]{};
\node [above] at ($({-cos(78.75)},{sin(78.75)})$) {$v_{n-3}$};
\draw ($({-cos(78.75)},{sin(78.75)})$) node(vn-3)[inner sep=0.3ex,circle,fill=black]{};
\node [above] at ($({cos(78.75)},{sin(78.75)})$) {$v_{n-4}$};
\draw ($({cos(78.75)},{sin(78.75)})$) node(vn-4)[inner sep=0.3ex,circle,fill=black]{};
\draw ($({cos(11.25)},{sin(11.25)})$) node(vn-2k)[inner sep=0.3ex,circle,fill=black]{};
\node [right] at ($({cos(11.25)},{sin(11.25)})$) {$v_{n-2k}$};
\draw ($({cos(33.75)},{sin(33.75)})$) node(vn-2k+1)[inner sep=0.3ex,circle,fill=black]{};
\node [above right] at ($({cos(33.75)},{sin(33.75)})$) {$v_{n-2k+1}$};


\draw ($({-cos(45)},{sin(45)-1})$) node(v2)[inner sep=0.3ex,circle,fill=black]{};
\node [above left] at ($({-cos(45)},{sin(45)-1})$) {$v_2$};
\draw (-1,-1) node(v3)[inner sep=0.3ex,circle,fill=black]{};
\node [left] at (-1,-1) {$v_3$};
\draw ($({-cos(45)},{-sin(45)-1})$) node(v4)[inner sep=0.3ex,circle,fill=black]{};
\node [below left] at ($({-cos(45)},{-1-sin(45)})$) {$v_4$};
\draw ($({cos(45)},{-1-sin(45)})$) node(vn-2k-3)[inner sep=0.3ex,circle,fill=black]{};
\node [below right] at ($({cos(45)},{-1-sin(45)})$) {$v_{n-2k-3}$};
\draw (1,-1) node(vn-2k-2)[inner sep=0.3ex,circle,fill=black]{};
\node [right] at (1,-1) {$v_{n-2k-2}$};
\draw ($({cos(45)},{-1+sin(45)})$) node(vn-2k-1)[inner sep=0.3ex,circle,fill=black]{};
\node [above right] at ($({cos(45)},{-1+sin(45)})$) {$v_{n-2k-1}$};
\draw (0,0) node(v1)[inner sep=0.3ex,circle,fill=black]{};
\node [above right] at (0,0.1) {$v_1$};
\node at ($({0.75*cos(55.5)},{0.75*sin(55.5)})$) {$\ddots$};
\node at (0,-1.75) {$\cdots$};

\draw (v1) edge (vn);
\draw (v1) edge (vn-4);
\draw (v1) edge (vn-3);
\draw (v1) edge (vn-2);
\draw (v1) edge (vn-1);
\draw (v1) edge (v2);
\draw (v1) edge (vn-2k+1);
\draw (v1) edge (vn-2k-1);
\draw (v1) edge (vn-2k);
\draw (vn-1) edge (vn);
\draw (vn-2k+1) edge (vn-2k);
\draw (v2) edge (v3);
\draw (v3) edge (v4);
\draw (vn-2k-2) edge (vn-2k-3);
\draw (vn-2k-1) edge (vn-2k-2);
\draw (v1) edge (v3);
\draw (v2) edge (v4);
\draw (vn-2k-1) edge (vn-2k-3);
\draw (v1) edge (vn-2k-2);
\draw (v2) edge (vn-2k-1);
\draw (vn-1) edge (vn-2);
\draw (vn-3) edge (vn-4);

\end{tikzpicture}
\caption{$n$ even}
\end{subfigure}
\caption{The constructions in Theorem \ref{main}. Edge multiplicities are omitted here for simplicity.}
\label{fig:main}
\end{figure}
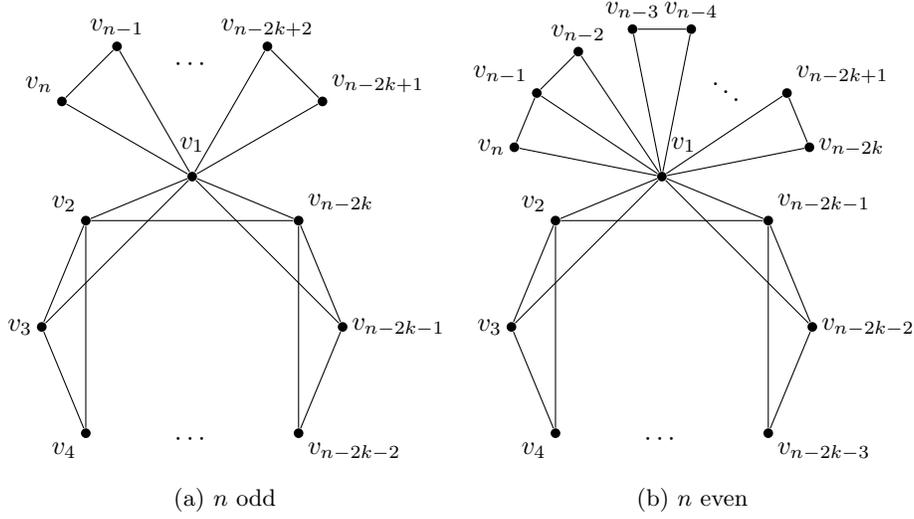
\begin{proof}[Proof of Theorem \ref{main}]
Recall that (\ref{decreasing}) and (\ref{even}) implies that $D=\sum_{i=1}^n(-1)^{i-1}d_i$ is a non-negative even integer. If $n=3$ or $n=4$, we are done by Lemma \ref{n=3,4}. Now assume $n\geq 5$. If $D\geq n-2$, we are done by Lemma \ref{bigD}. If $D\leq 4$ and $n$ is even, then $D=0,2,4$ and we are done by Lemma \ref{smallD}. If $D\leq 4$ and $n$ is odd, then since $D=(d_1-d_2)+\cdots+(d_{n-2}-d_{n-1})+d_n\geq 4$, we must have $D=4$ and so we are done by Lemma \ref{smallD} as well. Hence, it suffices to consider the case when $n\geq5$ and $6\leq D\leq n-3$, which can only happen if $n\geq 9$. Let $k=\frac12(D-4)$ and $d'_i=d_i-4$ for all $i\in[n]$. Note that $k\geq 1$ and $n-2k\geq 7$. Again, the triangular multigraph $G$ we construct differs slightly depending on the parity of $n$.

If $n$ is odd, let $(a_1,a_{n-2k+1},a_{n-2k+2},\cdots,a_n)$ and $(b_1,b_2,\cdots,b_{n-2k})$ be degree sequences defined as follows. $a_1=2k+\sum_{i=n-2k+1}^n(-1)^id_i$ and $a_i=d_i$ for all $n-2k+1\leq i\leq n$. $b_1=4+\sum_{i=2}^{n-2k}(-1)^id_i$ and $b_i=d_i$ for all $2\leq i\leq n-2k$. Then $a_{n-2k+1}\geq\cdots\geq a_n\geq4$, $a_1+\sum_{i=n-2k+1}^na_i$ is even, $a_1\leq\sum_{i=n-2k+1}^n(a_i-1)$ and $a_1+\sum_{i=n-2k+1}^n(-1)^{i-1}a_i=2k\geq(2k+1)-2$. Thus, by Lemma \ref{bigD}, there exists a triangular multigraph $G_1$ on vertices $v_1,v_{n-2k+1},\cdots,v_n$ with degree sequence $(a_1,a_{n-2k+1},\cdots,a_n)$. We also have $b_2\geq\cdots\geq b_{n-2k}\geq 4$, and $\sum_{i=1}^{n-2k}(-1)^{i-1}b_i=4$. Thus, by Lemma \ref{smallD}, there exists a triangular multigraph $G_2$ on vertices $v_1,\cdots,v_{n-2k}$ with degree sequence $(b_1,\cdots,b_{n-2k})$. Let $G=G_1\cup G_2$. Since $a_1+b_1=2k+4+\sum_{i=2}^{n}(-1)^id_i=2k+4+d_1-D=d_1$, $G$ is a multigraph with vertices $v_1,\cdots,v_n$ and degree sequence $(a_1+b_1,b_2,\cdots,b_{n-2k},a_{n-2k+1},\cdots,a_n)=(d_1,\cdots,d_n)$. Moreover, $G$ is triangular as both $G_1,G_2$ are and they only share a single vertex $v_1$. This proves the odd case. 

If $n$ is even, let $(a_1,a_{n-2k},a_{n-2k+1},\cdots,a_n)$ and $(b_1,b_2,\cdots,b_{n-2k-1})$ be degree sequences defined as follows. $a_1=2k+\sum_{i=n-2k}^n(-1)^id_i$ and $a_i=d_i$ for all $n-2k\leq i\leq n$. $b_1=4+\sum_{i=2}^{n-2k-1}(-1)^id_i$ and $b_i=d_i$ for all $2\leq i\leq n-2k-1$. Then $a_{n-2k}\geq\cdots\geq a_n\geq4$, $a_1+\sum_{i=n-2k}^na_i$ is even, $a_1\leq\sum_{i=n-2k}^n(a_i-1)$ and $a_1+\sum_{i=n-2k}^n(-1)^{i-1}a_i=2k\geq(2k+2)-2$. Thus, by Lemma \ref{bigD}, there exists a triangular multigraph $G_1$ on vertices $v_1,v_{n-2k},\cdots,v_n$ with degree sequence $(a_1,a_{n-2k},\cdots,a_n)$. We also have $b_2\geq\cdots\geq b_{n-2k-1}$ and $\sum_{i=1}^{n-2k-1}(-1)^{i-1}b_i=4$. Thus, by Lemma \ref{smallD}, there exists a triangular multigraph $G_2$ on vertices $v_1,\cdots,v_{n-2k-1}$ with degree sequence $(b_1,\cdots,b_{n-2k-1})$. Let $G=G_1\cup G_2$. Since $a_1+b_1=2k+4+\sum_{i=2}^{n}(-1)^id_i=2k+4+d_1-D=d_1$, $G$ is a multigraph with vertices $v_1,\cdots,v_n$ and degree sequence $(a_1+b_1,b_2,\cdots,b_{n-2k-1},a_{n-2k},\cdots,a_n)=(d_1,\cdots,d_n)$. Moreover, $G$ is triangular as both $G_1,G_2$ are and they only share a single vertex $v_1$. This proves the even case. 
\end{proof}

\section{Acknowledgements}
We would like to thank Domenico Mergoni Cecchelli for bringing Conjecture \ref{conj} to our attention during a workshop, and Amedeo Sgueglia, Kyriakos Katsamaktsis and Shoham Letzter for organising said workshop in University College London.

\bibliographystyle{abbrv}
\bibliography{bibliography}

\end{document}